\documentclass[11pt,a4paper]{amsart}
\usepackage{amsthm,amsmath,amssymb}
\newtheorem{theorem}{Theorem}
\newtheorem{lemma}{Lemma}
\newtheorem{proposition}{Proposition}
\newtheorem{corollary}{Corollary}
\newtheorem{remark}{Remark}

\newtheorem{assa}{Assumption}

\newtheorem{azero}{Assumption}

\newtheorem{aone}{Assumption}

\newtheorem{atwo}{Assumption}

\newtheorem{athree}{Assumption}

\DeclareMathOperator*{\esssup}{ess\,sup}

\begin{document}

\title[A Nonlinear Evolution Equation in an Ordered Banach Space]{A Nonlinear Evolution Equation in an Ordered Banach Space, Reconsidered}
\date{}

\author{Cecil P. Gr\"unfeld$^\dag$}
\thanks{$^\dag$ "Gheorghe Mihoc - Caius Iacob" Institute of Mathematical Statistics and Applied Mathematics of the Romanian Academy, P.O.Box 1-24, Bucharest, Romania E-mail address: grunfeld@ismma.ro; grunfeld@spacescience.ro}
\begin{abstract}

Results of a previous paper [Commun. Contemp. Math., {\bf 09} (2007) 217--251] on the existence of solutions to a nonlinear evolution equation in an abstract Lebesgue space, arising from kinetic theory, are re-obtained in the more general setting of a  real ordered Banach space, with additive norm on the positive cone,  which is not necessarily a (Banach) lattice.

In addition, an easily correctable technical error in the aforementioned paper is pointed out, and  repaired.
\end{abstract}
\maketitle


\section{Introduction}\label{sone}
This note is concerned with the Cauchy problems 
\begin{equation}
\frac{d}{dt}f(t)=Q^{+}(t,f(t))-Q^{-}(t,f(t)),\quad f(0)=f_0\in X_{+}, \quad t \geq 0,  \label{3}
\end{equation}
and 
\begin{equation}
\frac{d}{dt}f(t)=Af+Q^{+}(t,f(t))-Q^{-}(t,f(t)),\quad f(0)=f_0\in X_{+},\quad t \geq 0,  \label{a1}
\end{equation}
 in a  separable, ordered  real Banach space $X$, whose positive cone, $X_{+}$, is closed and generating,  and  whose norm is additive on $X_{+}$\footnote{ In some papers, such a Banach space is called "abstract state space" (see, e.g., \cite[p. 30]{Davies}, \cite{Arl1}, \cite{Arl2}.}. In (\ref{3}) and (\ref{a1}),  $f$ is defined from  ${\mathbb  R}_{+}:=[0,\infty)$ to $X_{+}$. Here,  $Q^{+}$ and $Q^{-}$  are   (nonlinear) mappings from ${\mathbb  R}_{+}\times {\mathcal D}$ to $X_{+}$, for some  ${\mathcal D}\subset X_{+}$  dense in $X_{+}$. It is assumed that for almost all (a.a.) $t\in {\mathbb  R}_{+}$, with respect to the Lebesgue measure on ${\mathbb  R}_{+}$, the operators $Q^{\pm}(t,\cdot)$ are isotone (i.e., order preserving) from ${\mathcal D}$ to $X_{+}$. 
 Moreover, in (\ref{a1}),  $A$ is the infinitesimal generator of a $C_{0}$ group of positive linear isometries on $X_{+}$.

Under the additional assumption that   $X$ is an abstract Lebesgue space (AL-space), i.e. a Banach lattice whose lattice norm is additive on the positive cone,   problems of the form (\ref{3}) and (\ref{a1}) were studied in Ref. \cite{Gr2007}, in the case of an abstract nonlinear evolution equation   arising from collisional kinetic theory. By construction, the model introduced in \cite{Gr2007} summarizes  common monotonicity properties (with respect to the order) of  various Boltzmann-like  kinetic equations and related models, and  satisfies a   generalization of  the  so-called Povzner inequality  \cite{Pov}, \cite{Ar}.

The main results of  \cite{Gr2007}, Theorem 3.1,  and  Corollary 3.1, consisted in proving  existence and uniqueness of  solutions to the above Cauchy problems. In addition,  Ref.  \cite{Gr2007} included  applications to the Boltzmann  equation with inelastic collisions and chemical reactions, a Povzner-like model with dissipative collisions, and Smoluchowski’s coagulation equation.

The present note is based on the  observation that  the lattice condition assumed in  \cite{Gr2007} was  needed only at one point of the analysis developed therein, specifically, in the proof of   \cite[Lemma 2.1]{Gr2007}. Therefore, extending the validity of  the  lemma, by removing the lattice assumption,  could  allow generalizing   results of \cite{Gr2007}.
Here, our goal is to prove that the main statements of \cite{Gr2007}  remain valid  in  the more general setting detailed in the beginning of this section, without imposing that $X$ should be a (Banach) lattice. This may be  of interest in  applications to problems involving Banach spaces which are not necessarily  lattices, e.g., the (real) space of  self-adjoint trace class operators over a  separable Hilbert space (endowed with the trace norm, and ordered by  the natural order of bounded self-adjoint operators), encountered in quantum kinetic modeling.

Once  results analogous to those of \cite[Lemma 2.1]{Gr2007} are proved  in our new setting,  the rest of the analysis of \cite{Gr2007} can be adapted  to  obtain the generalizations of \cite[Theorem 3.1 and Corollary 3.1]{Gr2007}.   Nevertheless,  for the sake of self-containment, in the present note, we outline the entire argument behind the  above generalizations, including some  clarifications and simplifications with respect to the presentation of \cite{Gr2007}.

Here we should  emphasize the fact\footnote {not mentioned elsewhere in literature, as far as we are aware.}  that, although Theorem 3.1  in Ref. \cite{Gr2007} is true (as stated therein), the proof provided in \cite{Gr2007} for assertion (a) of the theorem is  incomplete, because of an error, fortunately,  easily correctable in the context of  Ref. \cite{Gr2007}, as is explained in the last section and in the Appendix to the current paper.

Besides this Introduction,  our note contains four more sections and the above mentioned  Appendix. In Section 2, we detail  the general setting of the analysis to be developed in the rest of the paper, and  we extend  results of \cite[Lemma 2.1]{Gr2007}. In Section 3,  we state our main results, which have similar formulations as in Ref. \cite{Gr2007}. Section 4  includes proofs.  In the last section,  Conclusions, we briefly compare the results of the present note with those of \cite{Gr2007}. Finally, we comment on the   error above referred to, leaving its correction to the "Corrigendum" placed in  Appendix.

\section{General Setting}\label{stwo}
As  mentioned in Introduction, in this paper, $X$ is a separable, ordered  real Banach space, with order and norm denoted by $\leq$ and
$\|\cdot\|$, respectively. We assume that the positive cone $X_{+}:=\{g: g\in X; 0\leq g \}$ of $X$ is closed and generating (i.e., $X=X_{+}-X_{+}$), and  the norm  $\|\cdot\|$
 is additive on  $X_{+}$, i.e.,
\begin{equation}
\left\| g+h\right\| = \left\| g\right\| +\left\| h\right\| \quad g,h \in
X_{+}.  \label{1a}
\end{equation}
The above conditions are satisfied by AL - spaces, but also by ordered Banach spaces which are not necessarily lattices,  e.g. the anti-lattice of self-adjoint trace class operators with the trace norm, mentioned in Introduction (for other examples, see \cite[p. 30 - 31]{Davies}).

By  (\ref{1a}),   $X_{+}$ satisfies the strong Levi property, i.e.,  every  increasing (i.e., $\leq $ directed)  norm-bounded positive sequence of $X_{+}$  is norm convergent (to some element of $X_{+}$, because $X_{+}$ is closed).

Recall some  usual definitions. A mapping $\Gamma:{\mathcal D} (\Gamma)\subset X\mapsto X$, with ${\mathcal D}(\Gamma)
\cap X_{+}\neq \emptyset$, is called  {\it positive} if
$0\leq \Gamma g$ for all $0\leq g \in {\mathcal D}(\Gamma)$. Further, $\Gamma:{\mathcal D}(\Gamma)\subset X\mapsto X$ is
called  {\it isotone (or monotone)} if it preserves the order, i.e.,
\[ \forall  g,h \in {\mathcal D}(\Gamma), \; 
g\leq h \implies \Gamma g\leq \Gamma h.
\]
Similar definitions are introduces for mappings between two different ordered (Banach) spaces, in particular between $X$ and ${\mathbb  R}$ (endowed with the usual order).

Property (\ref{1a}) implies that the norm is monotone,  i.e.,
\begin{equation}
0\leq h \leq g \Rightarrow   \left\| h\right\| \leq\left\| g\right\|.   \label{1amonot}
\end{equation}

We will also use the following two definitions of \cite{Gr2007}. A set $\emptyset \neq {\mathcal M}\subset X$ is called  positively saturated (p-saturated)  if, for all $h\in  {\mathcal M}$ and $g\in X_{+}$, 
\[
g\leq h \Rightarrow g\in  {\mathcal M}.
\]
An operator $\Gamma :{\mathcal D}(\Gamma )\subset X\mapsto X$ is called closed with respect to the order (o-closed) if for every increasing sequence
$\left\{ g_n\right\} \subset {\mathcal D}(\Gamma )$ converging (in symbols, $\nearrow $) to $g$ in $X$, such that  $ \Gamma g_n  \rightarrow h\in X$, as $n \rightarrow \infty$, one has $g\in{\mathcal D}(\Gamma )$ and $\Gamma g=h$.
Obviously, a closed  isotone mapping  is  also o-closed.

Recall that if the set ${\mathfrak S} \subset {\mathbb R}$ is (Lebesgue) measurable and $g:{\mathfrak S} \mapsto X_{+} $  is   Bochner integrable, then
\[
\int_{\mathfrak S}g(s)ds \in X_{+},
\]
where $ds$ is the Lebesgue measure on the real line.

In our setting,  if $g:{\mathfrak S}\mapsto X_{+} $  is   Bochner integrable, then,  in view of (\ref{1a}),
\begin{equation}
\left\| \int_{{\mathfrak S}}g(s)ds\right\| =\int_{{\mathfrak S}}\left\|
g(s)\right\| ds,  \label{bi}
\end{equation}
$\forall$ ${\mathfrak S} \subset {\mathbb  R}$ measurable (the integral of the right hand side (r.h.s.) of  (\ref{bi}) being in the sense of Lebesgue).

According to Hille's theorem \cite[Theorem 3.7.12,   p. 83 ]{HiPh}, if $\Gamma:{\mathcal D}(\Gamma)\subset X
\mapsto X$ is a closed  linear operator, ${\mathfrak S} \subset {\mathbb  R}$ is measurable,   $g:{\mathfrak S} \mapsto {\mathcal D}(\Gamma)$ is Bochner integrable, and $\Gamma g$ is also Bochner integrable, then
\begin{equation}\label{com_boc}
\Gamma \int_{\mathfrak S} g(s)ds=\int_{\mathfrak S} \Gamma g(s)ds.
\end{equation}
We denote by  $L^{1}({\mathbb R}_{+};X_{+})$ ($L^{1}_{loc}({\mathbb R}_{+};X_{+})$)  the  space of equivalent classes of Bochner integrable  (locally Bochner integrable) maps from ${\mathbb R}_{+}$ to $X_{+}$. Also, $L^{\infty}_{loc}({\mathbb R}_{+};X_{+})$ denotes the space of equivalent classes of Lebesgue measurable maps $g:{\mathbb R}_{+}\mapsto X_{+}$ such that $||g(\cdot)||$ is  locally essentially bounded on ${\mathbb R}_{+}$.

Recall that if $\left\{S^t\right\} _{t\geq 0}$ is  a $C_0$ semigroup  on $X$, then  its infinitesimal generator $G:{\mathcal D}(G) \subset X  \mapsto X$ is a  closed linear operator, with the domain ${\mathcal D}(G)$ dense in $X$. The same is true for the positive integral powers $G^k$ (defined by $G^1:= G$,    ${\mathcal D}(G^k):=\{g: g\in {\mathcal D}(G^{k-1}),\; G^{k-1}g \in {\mathcal D}(G) \}$, $G^kg:=G (G^{k-1} g)$,  $k=2,3...$).

It is known that for every  real number $t>0$, and each $k=0,1,...$, 
\begin{equation}\label{rezol1}
\int_{0}^{t}S^s g ds \in  {\mathcal D}(G^{k+1}), \quad \forall g \in {\mathcal D}(G^k)
\end{equation}
(with the notations   $G^{0}:=I$,  ${\mathcal D}(G^0):=X$, where $I$ is the identity operator on $X$).   

One also knows  that  ${\mathcal D}(G^\infty) :=\cap _{k=1}^\infty {\mathcal D}(G^k)$  is  dense in $X$. Indeed, following,  e.g., \cite[Theorem 10.3.4, p. 308] {HiPh}), let   $\varphi:{\mathbb  R}_{+} \mapsto {\mathbb  R} $, indefinitely differentiable on $(0,\infty)$, with compact support, and satisfying,
$
\int_0^\infty \varphi (t)dt=1.
$
Then  every  $g\in X$ can be approximated by  a sequence  ${g}_n\rightarrow g$ as $n\rightarrow \infty $, where
\begin{equation}
{\mathcal D} (G^\infty )\ni{g}_n:=n\int_0^\infty \varphi
(nt)S^t g dt,\quad n=1,2,... . \label{has}
\end{equation}

We recall that a { \it positive} $C_0$ semigroup on $X$ is
a $C_0$ semigroup on $X$, which leaves the cone $X_{+}$ invariant. If $G$ denotes its infinitesimal generator, then ${\mathcal D}_{+} (G^\infty) :={\mathcal D}(G^\infty) \cap X_{+}$  is dense in $X_{+}$, as can be immediately seen by choosing $\varphi  \geq 0$  and $g\in X_{+}$ in (\ref{has}).   In particular, ${\mathcal D}_{+}(G^k):={\mathcal D}(G^k)\cap X_{+}$ is dense in $X_{+}$ for all $k=1,2,...$.

The next  lemma includes  simple but useful facts stated in \cite[Lemma 2.1]{Gr2007},  which actually hold under more general conditions than those assumed in the  lattice setting of Ref \cite{Gr2007}. 
\begin{lemma} \label{lema3} Let $\{S^{t}\}_{t\geq 0}$ be a positive semigroup on $X$ with infinitesimal generator $(-G)$. Suppose that there exists some number  $\gamma > 0$ such that
	\begin{equation}
	\gamma g \leq Gg ,  \quad \forall g \in{\mathcal D}_{+}(G).
	\label{gamarez}
	\end{equation}
	Then:
	
(a) For each $g\in X_{+}$,
	\begin{equation}
	0\leq S^t g\leq \exp (-\gamma t)g, \quad \forall  t\geq0.   \label{7}
	\end{equation}
	
(b) For each $g\in X_{+}$, there exists an increasing sequence $\left\{g_n\right\} \subset {\mathcal D}_{+}(G^\infty) $, such that $g_n\nearrow g$ as
	$n\rightarrow \infty $.

(c) Let $p$ be a positive integer. If $\{g_n\} \subset  {\mathcal D}_{+}(G^p)$ is increasing and $\{G^p g_n\}$ is norm bounded, then there exists $g\in {\mathcal D}_{+}(G^p)$ such that  $G^k g_{n} \nearrow G^k g$ for all $k=0,1,.., p$.

(d) The sets ${\mathcal D}_{+}(G ^k)$, $k=1,2,...$, and  ${\mathcal D}_{+}(G^\infty)$ are
p-saturated.
\end{lemma}
\begin{proof} (a) If $ g \in {\mathcal D}_{+}(G )$,
	then
$ \frac{d}{dt}S^t g = - G S^t g$. But  (\ref {gamarez}) implies $ \frac{d}{dt}S^t g \leq  -\gamma S^t g$,  which yields (\ref{7}) (in the case $ g \in {\mathcal D}_{+}(G )$).  Since $ {\mathcal D}_{+}(G )$ is dense in $X_{+}$,  it follows that for every $g\in X_{+}$,  there is some sequence ${\mathcal D}_{+}(G ) \ni{g}_n \rightarrow g$, as $n\rightarrow \infty$,  with the property  $0\leq S^t g_{n}\leq \exp (-\gamma t)g_{n}$ for all $n$. Then the limit satisfies (\ref{7}), because $ X_{+}$ is closed.

(b) Let   $\varphi$   be  positive in  (\ref{has}). Then  a simple computation starting from (\ref{has}), and   making use of (\ref{7}), implies easily $0\leq g_n  \leq g_{n+1}$ for all $n=1,2,...$. 

(c) By Levi's property, there exists $u_{p}\in X_{+}$ such that $G^p g_n\nearrow u_p$, as $\rightarrow \infty$. Observe that,  by (\ref{gamarez}) and (\ref{1amonot}), one has $\|G^k g_{n}\|\leq  \gamma^{k-p}\|G^{p}g_{n}\|$ for all $k=0,1,2,...,p$. Then,  for each $k=0,1,2,...,p$, the sequence  $\{G^k g_{n}\}$ is norm bounded. Therefore,  by Levi's property, there exists $u_{k}\in X_{+}$ such that  $u_{k,n}:=G^k g_{n} \nearrow  u_{k}$ as $n \rightarrow \infty$. For $k=0$ and $k=1$,  we have $g_{n} \nearrow  u_{0}$ and $u_{1,n}=G g_{n} \nearrow  u_{1}$, respectively, as $n \rightarrow \infty$. But $G$ is closed. Consequently, $u_{0}\in {\mathcal D}_{+}(G )$  and $G u_{0}=u_{1}$. To complete the proof of (c), we  proceed inductively, using that $G^k$ is closed, $k=2,3,...$.

(d). Let  $k=1,2,..,$  be fixed, and  $0\leq g \leq h\in {\mathcal D}_{+}(G^{k})$. By  virtue of (b),  there is a  sequence   ${\mathcal D}_{+}(G^\infty) \ni g_{n} \nearrow g \leq h$. Therefore, $G^k g_{n}$ is increasing and $\|G^k g_{n}\| \leq \|G^k h\|$ for all $k=1,2,...$. Then  (c) applies, hence $g \in{\mathcal D}_{+}(G^{k})$. 
\end{proof}

We remark that the above lemma does not use the additivity of the norm. Moreover, Levi's property is not needed in the proof of (a) and (b).
\section{Model and main result}
  In the setting detailed in the previous sections, we investigate  (\ref{3}) and (\ref{a1}), by assuming the hypotheses of the  model introduced in  Ref. \cite{Gr2007}, as follows. 
  
  I. General assumptions:

$\bullet$ The mappings ${\mathbb  R}_{+} \ni t\mapsto Q^{\pm}(t,g(t)) \in X_{+}$ are   (Lebesgue) measurable\footnote{We do not distinguish between strong and weak measurability,  because $X$ is separable.} for every measurable $ g:{\mathbb  R}_{+}\mapsto X_{+}$ which satisfies $g(t) \in {\mathcal D}$   almost everywhere (a.e.) on ${\mathbb  R}_{+}$.

$\bullet$  For a.a. $t\geq 0$, the positive operators $Q^{\pm }(t,\cdot )$ are isotone and   o-closed, and their common domain ${\mathcal D}$ is p-saturated. 

II. Specific assumptions:
\begin{assa}
There exists a linear operator $\Lambda :{\mathcal D}(\Lambda )\subset X\mapsto X$ such that $(-\Lambda )$ is the infinitesimal generator of a positive $C_0$ semigroup on $X$, with the properties   ${\mathcal D}_{+}(\Lambda)\subset {\mathcal D}$,  $Q^{\pm}(t,{\mathcal D}_{+}(\Lambda ^k))\subset{\mathcal D}_{+}(\Lambda ^{k-1})$, $t\geq 0$ a.e., $k=2,3$, and:
\end{assa} 

\begin{azero}
There exists a number $\lambda_0>0$ such that
\begin{equation}
\lambda _{0} g \leq \Lambda g, \quad \forall g\in {\mathcal D}_{+}(\Lambda).
\label{lb}
\end{equation}
\end{azero}
\begin{aone}
There exists a  positive, non-decreasing, convex function 
$a:{\mathbb R}_{+}\mapsto {\mathbb  R}_{+}$, such that for a.a.  $t\geq 0$, 
\begin{equation}
0 \leq  Q^{-}(t,g) \leq a(\left\| \Lambda g\right\|) \Lambda g,
\quad  \forall g\in {\mathcal D}_{+}(\Lambda),
\label{a2poz}
\end{equation}
and  the mapping
${\mathcal D}_{+}(\Lambda)\ni g\mapsto a(\left\| \Lambda g\right\|
)\Lambda g-Q^{-}(t,g)$ $\in X$ is {\it isotone}.
\end{aone}
\begin{atwo}
For a.a. $t\geq 0$,
\begin{equation}
\Delta (t,g):=\left\Vert \Lambda Q^{-}(t,g)\right\Vert -\left\Vert
\Lambda Q^{+}(t,g)\right\Vert\geq 0,
\quad \forall g\in {\mathcal D}_{+}(\Lambda^2),  \label{3b}
\end{equation}
and the map ${\mathcal D}_{+}(\Lambda^2)\ni g\mapsto \Delta
(t,g)\in {\mathbb  R}_{+}$ is {\it isotone}.
\end{atwo} 	
\begin{athree}
 There exist a  positive non-decreasing function  $\rho :{\mathbb  R}_{+}\mapsto
{\mathbb   R }_{+}$  and a closed, positive, linear operator $\Lambda
_{1}:{\mathcal D}(\Lambda _{1})\subset X_{+}\mapsto X_{+}$, with
${\mathcal D}_{+}(\Lambda)\subset {\mathcal D}(\Lambda
_{1})$, such that for a.a. $t\geq 0$,   
\begin{equation}\label{lambdaunu}
\left\Vert \Lambda_{1} Q^{-}(t,g)\right\Vert -\left\Vert \Lambda_{1}
Q^{+}(t,g)\right\Vert\geq 0, \quad \forall g\in {\mathcal D}_{+}(\Lambda^2),
\end{equation}
\begin{equation}
\left\Vert \Lambda^{2} Q^{+}(t,g)\right\Vert \leq \left\Vert \Lambda^{2}
Q^{-}(t,g)\right\Vert + \rho (\left\Vert \Lambda _{1}g\right\Vert
)\left\Vert \Lambda ^{2}g\right\Vert, \quad \forall g\in {\mathcal D}_{+}(\Lambda^3).  \label{povg}
\end{equation}
\end{athree}

Some comments and remarks are in order.

The convexity assumed in $(A_{1})$  implies that the function $a$  is continuous on $(0,\infty)$, and its derivative is a.e. defined,  positive and non-decreasing on  ${\mathbb R}_{+}$.   In our case, $a(0)=a(0+)$, because $a$ is positive and non-decreasing. So,  the usual representation of $a$,  in terms of its derivative, takes the form 
\begin{equation}\label{arep}
a(x)=a(0)+\int_{0}^{x} a'(\xi) d\xi, \quad x\in {\mathbb R}_{+}.
\end{equation}

By the above assumptions,  for each $k=1,2,...$, the linear operator  $\Lambda^k$ is positive, closed, and densely defined. Besides,  ${\mathcal D}_{+}(\Lambda^k)$, $k=1,2....$, and 
${\mathcal D}_{+}(\Lambda^\infty):= {\mathcal D}_{+}(\Lambda^{\infty})\cap X_{+} $   are   p-saturated, because  $\Lambda$ fulfills the conditions of Lemma  \ref{lema3}.

Note here that,  since ${\mathcal D}_{+}(\Lambda^k)$  is  p-saturated (for all $k$), we get from (\ref{a2poz})
\begin{remark}\label{rdinifty}
	For each  $k=1,2,..$,  and for a.a. $t \geq 0$, one has  $Q^{-}(t, {\mathcal D}_{+}(\Lambda^k)))$ $ \subset$ $ {\mathcal D}_{+}(\Lambda^{k-1}))$, and
\begin{equation}\label{aQ}
\Lambda^{k-1}Q^{-}(t,g) \leq a(\|\Lambda g\|)\Lambda^{k}g \quad, \forall g\in {\mathcal D}_{+}(\Lambda^k)
\end{equation}
(where $\Lambda^0=I$).	
In particular, $Q^{-}(t, {\mathcal D}_{+}(\Lambda^\infty) ) \subset {\mathcal D}_{+}(\Lambda^\infty)$ a.e. on ${\mathbb  R}_{+}$.
\end{remark}
Remark \ref{rdinifty} shows  that the inclusion conditions on $Q^{\pm}(t,{\mathcal D}_{+}(\Lambda^k))$  imposed in the beginning of this section are implicitly fulfilled by  $Q^{-}$, in the context of Assumption ($A_1$).\footnote{However,  we kept those conditions, in order to have a priori well-defined statements in Assumptions $A_2$ and $A_3$.}

The above model assumptions indicate some control on $Q^{\pm}$, in terms of powers of $\Lambda$.
Specifically, assumption ($A_1$) shows that, although  $Q^{+}$  might exhibit a highly nonlinear behavior,  it remains somehow controlled by the linear operator $\Lambda$  on  each set $\{g:\in {\mathcal D}_{+}(\Lambda): \|\Lambda g \|=constant\}$.
By ($A_{2}$) one controls  $\Delta$,  in terms of $\Lambda$ and $\Lambda^2$.  Indeed, by (\ref{3b})
for a.a. $t \geq 0$,
\begin{equation}\label{deltamarg}
0 \leq \Delta(t,g)\leq \|\Lambda Q^{-}(t,g)\|\leq a(\|\Lambda g\|)\|\Lambda^2 g\|, \quad \forall g\in {\mathcal D}_{+}(\Lambda ^2).
\end{equation}

Observe that  (\ref{lb}) implies immediately
\begin{equation} \label{541}
\left\| g\right\| \leq \lambda _0^{-1}\left\| \Lambda g\right\|,    \quad \forall g\in {\mathcal D}_{+}(\Lambda).
\end{equation}
By applying (\ref{541}) to $Q^{\pm }(t,g)$,  and using (\ref{deltamarg}), one finds
\begin{equation}
\left\| Q^{\pm }(t,g)\right\| \leq \lambda _0^{-1}\left\| \Lambda Q^{\pm
}(t,g)\right\| \leq \lambda _0^{-1}\left\| \Lambda Q^{-}(t,g)\right\|
\leq a(\left\| \Lambda g\right\| )\lambda _0^{-1}\left\| \Lambda
^2g\right\|  \label{a54}
\end{equation}
a.e on ${\mathbb  R}_+$ for all $g\in {\mathcal D}_{+}(\Lambda ^2)$.

	Property  (\ref{a54}) entails
\begin{equation}\label{qzero}
	Q^{\pm }(t,0)=0\quad and \quad \Delta (t,0)=0\quad  a.e.\quad on
\quad	{\mathbb  R}_{+}.
\end{equation}

It is useful to introduce the  following  spaces of "abstract  moments of $\Lambda$ ". For $k=1,2,... $, let $L^1_{k, loc}({\mathbb R}_{+};X_{+})$ denote the space (of equivalent classes) of  measurable maps $g:{\mathbb R}_{+} \mapsto \mathcal{D}(\Lambda^k)$, satisfying  $\Lambda^k g$ $\in$   $L^1_{loc}({\mathbb R}_{+};X_{+})$.  We  set $L^1_{0, loc}({\mathbb R}_{+};X_{+})$ $:=$  $L^1_{loc}({\mathbb R}_{+};X_{+})$ and 
$L^1_{\infty, loc}({\mathbb R}_{+};X_{+})$ $:=$ $\cap _{k=1}^{\infty} L^1_{k, loc}({\mathbb R}_{+};X_{+})$.  Also, by $L^\infty_{2, loc}({\mathbb R}_{+};X_{+})$, we denote the space  of measurable maps $g:{\mathbb R}_{+} \mapsto \mathcal{D}(\Lambda^2)$, satisfying  $\Lambda^2 g$ $\in$   $L^\infty_{loc}({\mathbb R}_{+};X_{+})$. 

We consider the above spaces, with the natural order of $X_{+}$,  (i.e., induced by the order $\leq$ of $X$).

Due to  (\ref{lb}) and (\ref{1amonot}), clearly,  $L^\infty_{2, loc}({\mathbb R}_{+};X_{+})$ $\subseteq$ $L^1_{2, loc}({\mathbb R}_{+};X_{+})$, and
\begin{remark}\label{lem1a} If  $k,p=0,1,2,..., \infty$, and $k \leq p$, then $L^1_{p, loc}({\mathbb R}_{+};X_{+})$ $\subseteq$ $L^1_{k, loc}({\mathbb R}_{+};X_{+})$. 
\end{remark}

Inequality (\ref{3b}) puts into  an abstract form common elements of various conservation/dissipation properties of  kinetic models (for details, see  Ref.  \cite{Gr2007}). In addition, (\ref{povg}) can be regarded as an abstract correspondent to the Povzner inequality \cite{Pov, Ar} (see also \cite{Gr2007}) .

Our results  are concerned with the existence and uniqueness of global (in time), strong solutions to  Eq. (\ref{3}) and   mild solutions to Eq. (\ref{a1}).

Here we say that $f$ is a (strong) solution to Eq. (\ref{3}) on ${\mathbb  R}_{+}$  if it is (strongly) absolutely continuous on ${\mathbb  R}_{+}$,  (strongly) differentiable a.e. on ${\mathbb  R}_{+}$, satisfies Eq. (\ref{3}) a.e. on ${\mathbb  R}_{+}$, and $f(0)=f_{0}$.

Note that
$f$ is a strong solution of Problem (\ref{3}) on ${\mathbb  R}_{+}$ iff
\begin{equation}
f(t)=f_0+\int_0^t [Q^{+}(s,f(s))-Q^{-}(s,f(t))]ds, \quad \forall t \geq 0 \label{3bis}
\end{equation}
(where the integral is in the sense of Bochner).

One can similarly define strong solutions to Eq. (\ref{a1}).

Let   $\{U^t\}_{t\in {\mathbb R}}$ denote the $C_{0}$ group of positive linear isometries,  defined by the the infinitesimal generator  $A$ introduced in  (\ref{a1}). Recall that  any strong solution of Eq. (\ref{a1}) satisfies
\begin{equation}
f(t)=U^tf_0+\int_0^tU^{t-s}[Q^{+}(s,f(s))-Q^{-}(s,f(s))]ds,  \quad t\geq 0,  \label{a1b}
\end{equation}
but the converse is not generally true.

Let   $ C({\mathbb  R}_{+};X_{+})$ be the space of continuous functions from ${\mathbb  R}_{+}$ to $X_{+}$.
We say that $f\in C({\mathbb  R}_{+};X_{+})$ is a  mild solution of Eq. (\ref
{a1}) on ${\mathbb  R}_{+}$, if it satisfies Eq. (\ref{a1b}) for all $t \geq 0$.

A statement similar to   \cite[Theorem 3.1]{Gr2007}  can be proved  in the more general setting of this paper.
\begin{theorem}
	\label{th1} Suppose $f_0\in {\mathcal D}_{+}(\Lambda^2)$, in  (\ref{3}).  Let either of the following hold:
	
	(a) $Q^{+}(t,{\mathcal D}_{+}(\Lambda^\infty) )$ $\subset $ ${\mathcal D}_{+}(\Lambda^\infty)
	$, $t\geq 0$ a.e., and  $\Lambda ^kQ^{+}(\cdot ,{\mathcal D}_{+}(\Lambda^\infty))$ $ \subset$ $ L^{1}_{loc}({\mathbb  R}_{+};X_{+})$, $k=1,2,...$.
	
	(b) The operators $Q^{\pm }$ do not depend explicitly on $t$.

	Then the Cauchy problem (\ref{3}) has  a unique positive  strong solution $f$ on ${\mathbb  R}_{+}$, such that $f(t)\in {\mathcal D}_{+}(\Lambda ^2)$ for all $t\geq 0 $, and $\|\Lambda^2 f(\cdot)\|$ is locally bounded on ${\mathbb R}_{+}$.
	Moreover, $\Lambda f\in C({\mathbb  R}_{+};X_{+})$. Furthermore,  $f$ satisfies 
	\begin{equation}
	\left\| \Lambda f(t)\right\|+\int_0^t\Delta
	(s,f(s))ds =\left\| \Lambda f_0\right\|, \quad \forall t\geq 0,  \label{5}
	\end{equation}
    and
	\begin{equation}
	\left\| \Lambda ^2f(t)\right\| \leq \left\| \Lambda ^2f_0\right\| \exp (\rho (\left\| \Lambda
	_1f_0\right\| )t), \quad \forall t \geq 0.
	\label{7c}
	\end{equation}
\end{theorem}

In Theorem \ref{th1}, conditions (a) and (b) do not mutually exclude each other. 

Formula (\ref{5}) generalizes a priori "conservation/ dissipation" estimates considered in, e.g., \cite{AnGr2} (For more details,  the reader is referred  to  \cite{Gr2007}).

In applications,  one may have $\Lambda_1=\Lambda $, when some  conditions of ($A_3$) become redundant (see \cite{Gr2007}).

The above theorem has an immediate consequence, with a similar statement as  \cite[Corollary 3.1]{Gr2007}.

Suppose that  $U^t{\mathcal D}(\Lambda )= {\mathcal D}(\Lambda )$ and $U^t{\mathcal D}(\Lambda _1)={\mathcal D}(\Lambda _1)$. Also assume that for each $t >0$, $U^t\Lambda
=\Lambda U^t$ on ${\mathcal D}(\Lambda )$ and 
 $U^t\Lambda_1=\Lambda _1U^t$
on ${\mathcal D}(\Lambda _1)$.
\begin{corollary}
	{\label{cor1} }Suppose $f_{0}\in {\mathcal D}_{+}(\Lambda^{2})$, in
	(\ref{a1}). Let $Q^{+}(t,  {\mathcal D}_{+}(\Lambda^\infty)) \subset
	{\mathcal D}_{+}(\Lambda^\infty)$   a.e. on ${\mathbb  R}_{+}$, and  $\Lambda ^kQ^{+}(\cdot ,U^
	{\cdot }g)$ $\in L_{loc}^1({\mathbb  R}_{+};X_{+})$ for all $g\in {\mathcal D}_{+}^
	\infty $, $k=1,2,....$.  Then the Cauchy problem (\ref{a1}) has  a unique positive global mild solution $f$ on ${\mathbb  R}_{+}$,
	such that $f(t)\in {\mathcal D}_{+}(\Lambda ^2)$ for all
	$t\geq 0$,  and $\|\Lambda^2 f(\cdot)\|$ is locally bounded on ${\mathbb R}_{+}$. Moreover,
	$\Lambda f\in C({\mathbb  R}_{+};X_{+})$. Furthermore, $f$ satisfies Eq. (\ref{5}) and
	inequality (\ref{7c}).
\end{corollary}
\begin{proof} (see  \cite[Corollary 3.1]{Gr2007}). The transformation    $F(t):=U^{-t}f(t)$,  simply reduces (\ref{a1b})  to
	\begin{equation}
	\frac d{dt}F(t)=Q^{+}_{U}(t,F)-Q^{-}_{U}(t,F(t)), \quad
	F(0)=f_0,  \quad  t\geq 0, \quad a.e., \label{a2}
	\end{equation}
where  	$Q^{\pm}_{U}(t,g):=U^{-t}Q^{\pm}(t,U^t g)$  for all $g \in {\mathcal D}$,  and  a.a. $t\geq 0$.  Then it is sufficient to check that Theorem  \ref{th1} applies with $Q^{\pm}$
replaced by $Q^{\pm}_{U}$.   	
\end{proof}
\section{Technical proofs}
As in  Ref. \cite{Gr2007}, we reduce (\ref{3}) to an equivalent problem  for an equation  more suitable for monotone iteration. To this end, consider the problem
\begin{equation}
\frac d{dt}f(t)+a(\left\| \Lambda f_0\right\| )\Lambda f=B(t,f,f), \quad
f(0)=f_0\in X_{+}, \quad  t\geq 0,  \label{4}
\end{equation}
and its associated integral form
 \begin{equation}
\begin{split}
f(t & )=f_0+\int_0^t [Q^{+}(s,f(s))-Q^{-}(s,f(s))]ds
\\
+ & \int_0^t\left[ a\left( \left\| \Lambda f(s)\right\| +\int_0^s\Delta (\tau,
f(\tau))d\tau \right) -a(\left\| \Lambda f_0\right\| )\right] \Lambda
f(s)ds \quad \forall t \geq 0,
\end{split}
\label{35}
\end{equation}
where $a$  is given by ($A_1$), and $B$ is formally defined  a.e. on $t$ $\in$ $R_{+}$, by
\begin{equation}
B(t,g,h):=Q^{+}(t,g(t)) - Q^{-} (t,g(t))+a\left( \left\| \Lambda g(t)\right\| +\int_0^t\Delta
(s,h(s))ds\right) \Lambda g(t),
\label{6}
\end{equation}
for, say,  any  $g, h$ $\in$ $L^\infty_{2,loc}(\mathbb R_{+};X_{+})$.

\begin{proposition}\label{cons}
	If $f$ is a positive strong solution to (\ref{3}) such that $f(t)\in {\mathcal D}_{+}(\Lambda ^2)$ for all $t\geq 0 $, and $\|\Lambda^2 f(\cdot)\|$ is locally bounded on ${\mathbb R}_{+}$, then  $\Lambda f\in C({\mathbb  R}_{+};X_{+})$ and $f$ satisfies (\ref{5}).		
\end{proposition}
\begin{proof}
a)	From (\ref{a54}) and the assumptions on $f$, we get $\Lambda f$,  $\Lambda Q^{\pm}(\cdot,f)$ $\in$ $L^1_{loc}({\mathbb R}_{+};X_{+})$. Then we simply find that  $\Lambda f\in C({\mathbb  R}_{+};X_{+})$ by applying  $\Lambda$ 	to (\ref{3bis}), and using (\ref{com_boc}). Moreover, playing conveniently with the terms in the resulting equality, we  apply (\ref{1a}) and (\ref{bi}), and, finally, take advantage of (\ref{3b}), to obtain (\ref{5}). 
\end{proof}
\begin{proposition}\label{equiv}
Let ${\mathbb R}_{+} \ni t\mapsto f(t)\in {\mathcal D}_{+}(\Lambda ^2)$ be such that $\|\Lambda^2 f(\cdot)\|$ is locally bounded on ${\mathbb R}_{+}$. Then $f$	is a strong solution to  (\ref{3}) iff it is a strong solution to (\ref{4}). 
\end{proposition}
\begin{proof}
Under the conditions of the proposition, if  $f$ 
is a strong solution to (\ref{3}), then it fulfills the conditions of Proposition \ref{cons}, so that,  in view of (\ref{5}), $f$ is also a strong solution to (\ref{4}). 
Conversely, suppose that $f$ is a strong solution of (\ref{4}), Then $f$  satisfies (\ref{35}), where applying (\ref{com_boc}) (with $\Gamma=\Lambda$), and writing conveniently the resulting equality, we get
\begin{equation*}
\begin{split}
\Lambda f_0&+\int_0^t \left[\Lambda Q^{+}(s,f(s)) +a\left( \left\| \Lambda f(s)\right\| +\int_0^s\Delta (\tau,
f(\tau)d\tau \right)\Lambda^{2}
f(s)\right] ds \\
= & \Lambda f(t) +\int_0^t
\left[ \Lambda Q^{-}(s,f(s)) +a(\left\| \Lambda f_0\right\|) \Lambda^{2}f(s)\right] ds, 
\end{split}
\end{equation*}
Further,   applying   (\ref {1a}) and  (\ref{bi}) in the above equality we obtain
 \begin{equation}
\psi (t)=\int_0^t\left[ a(\left\| \Lambda f_0\right\| )-a\left( \left\|
\Lambda f(s)\right\| +\int_0^s\Delta (\tau ,f(\tau))d\tau \right) \right]
\left\| \Lambda ^2f(s)\right\| ds,  \label{psi}
\end{equation}
where $\psi (t):=\left\| \Lambda f_0\right\| -\left\| \Lambda f(t)\right\|-\int_0^t\Delta (s,f(s))ds$. 
Consequently,
\begin{equation}\label{psiineq}
|\psi (t)|\leq\int_0^t\left| a(\left\| \Lambda f_0\right\| )-a\left( \left\|
\Lambda f(s)\right\| +\int_0^s\Delta (\tau ,f(\tau))d\tau \right) \right|
\left\| \Lambda ^2f(s)\right\| ds,
\end{equation}
Fix an arbitrary $0<T< \infty$. 
Since  $\|\Lambda^2 f(\cdot)\|$ is locally bounded on ${\mathbb R}_{+}$, 
by (\ref{lb}) and the positivity of $\Delta$, we can write   $\left\|\Lambda f(t)\right\| +\int_0^t\Delta (\tau ,f(\tau))d\tau \leq \alpha_{T,f}$ $\forall$ $t\in [0,T]$, where
$\alpha_{T,f}$:=$\lambda^{-1}[\esssup_{t\in[0, T]} \left\|\Lambda^2 f(t)\right\|] +\int_0^T\Delta (\tau ,f(\tau))d\tau <\infty $ depends  on  $T$ and $f$. In particular, $\|\Lambda f_{0}\| \leq \alpha_{T,f}$.
But the convexity (or, simply (\ref{arep}))  implies that $a$   is   Lipschitz  on $[0,\alpha_{T,f}]$.  Then there is  a number $\beta_{T,f}>0$ (depending on  $T$ and $f$) such that
\begin{equation}\label{lippsi}  
| a(\left\| \Lambda f_0\right\| )-a\left( \left\|
\Lambda f(s)\right\| +\int_0^s\Delta (\tau ,f(\tau))d\tau \right) 
| \leq \ \beta_{T,f} |\psi(s)|, \quad 0\leq s\leq T,
\end{equation}
which can be introduced in (\ref{psiineq}), and combined with the local boundedness of $\|\Lambda^2 f\|$, to obtain
\[
0\leq |\psi (t)|\leq \beta_{T,f} \int_0^t| \psi (s)|\left\| \Lambda ^2f(s)\right\| ds\leq
\gamma_{T,f}\int_0^t | \psi (s)|ds,  \quad \forall t \in [0, T],
\] where  $\gamma_{T,f}>0$ is also a number depending only on  $T$ and
$f$. Finally, the Gronwall's inequality yields  $\psi (t)=0$ for all $0\leq t\leq T$. This concludes the proof, because $T$ is arbitrary.
\end{proof}
Thus proving Theorem  \ref{th1} is equivalent to demonstrating the existence and uniqueness of positive solutions to (\ref{4}),   having the property (\ref{7c}). 

Before proceeding to the study of (\ref{4}) in the general case, observe that from (\ref{qzero}),  (\ref{5}),  and the properties of $\Lambda$ and $\Delta$, it results that, in the class of solutions considered in Theorem \ref{th1}, $f\equiv0$ is the only  solution of  Problem (\ref{3}) with initial datum $f_{0}=0$. 
Moreover, if  $0 \neq f_0\in {\mathcal D}_{+}(\Lambda^{2})$, $a(\|\Lambda f_0\|)=0$, and if  $f$  is strong solution of (\ref{3}), with $f(0)=f_{0}$,  and properties  as in Theorem \ref{th1},  then $\|\Lambda f(t)\|\leq \|\Lambda f_{0}\|$  on ${\mathbb R}_{+}$, due to (\ref{a54}). But $a$ is positive and non-decreasing. Therefore, $a(\|\Lambda f(t)\|)=0$ on ${\mathbb R}_{+}$.  Consequently,  (\ref{a54}) implies  $Q^{\pm}(t,f(t))\equiv 0$ a.e. on ${\mathbb R}_{+}$, which, introduced in (\ref{3}), yields  $f(t)\equiv f_{0}$ on ${\mathbb R}_{+}$.

Therefore, in the following, we  assume  $f_0 \neq 0$ and $a(\|\Lambda f_0\|)\neq0$.

The proof of Theorem \ref{th1} is close to the main argument of  \cite{Gr2007}, and relies on  the fact that any positive strong solution of  (\ref{4}) satisfies
\begin{equation} \label{6ab}
f(t)=V^tf_0+\int_0^tV^{t-s}B(s,f,f)ds, \quad t \geq 0
\end{equation}
(the integral being in the sense of Bochner), where  $\left\{ V^t\right\} _{t\geq 0}$  is the positive $C_0$ semigroup on $X$ with infinitesimal generator $L:=-a(\left\| \Lambda f_0\right\| )\Lambda$.
Thus the positive strong  solutions of  Problem (\ref{3}) can be found among the positive solutions of  Eq. (\ref{6ab}), which  satisfy (\ref{5}).

Note that $\left\{ V^t\right\} _{t\geq 0}$ fulfills  the conditions of Lemma \ref{lema3}. Then $\forall$  $g\in X_{+}$, 
\begin{equation}
0\leq V^t g\leq \exp (-\lambda_0 a(\left\| \Lambda f_0\right\|) t) g \leq g, \quad \forall t \geq 0.  \label{vt}
\end{equation}

Formally, $B(t,g,h)$ defined by  (\ref{6}) is positive. Indeed, applying (\ref{a2poz}) in  (\ref{6}), and  using  the monotonicity of $a$, and the positivity of $\Lambda$ and $\Delta$, we get 
$ 0 \leq $
$ Q^{+}(t,g(t)) +  \left( a\left(\| \Lambda g(t)\| +\int_0^t\Delta(s,h(s))ds\right) -a(\| \Lambda g(t) \|\right) \Lambda g(t)$ $ \leq B(t,g,h)$.
By  (\ref{a54}) and the   properties of $a$ (cf. ($A_1$)),  expression (\ref{6}) defines a mapping
\begin{equation} \label{Bmap}
 L^{\infty}_{2,loc}({\mathbb R}_{+}; X_{+})\times  L^{\infty}_{2,loc}({\mathbb R}_{+}; X_{+}) \ni (g,h)\mapsto B(\cdot, g,h)\in L_{1,loc}^1({\mathbb R}_{+}; X_{+}).
\end{equation}
\begin{lemma}\label{lem2}
	The mapping (\ref{Bmap}) is isotone\footnote{in the sense that if $(g_i, h_i)$  $\in$ $L^{\infty}_{2,loc}({\mathbb R}_{+}; X_{+})\times  L^{\infty}_{2,loc}({\mathbb R}_{+}; X_{+})$, $i=1,2$,
		and    $g_1(t)\leq g_2(t)$, $h_1(t)\leq h_2(t)$ a.e. on ${\mathbb R}_{+}$, then $B(t,g_1,h_1)\leq B(t,g_2,h_2)$ a.e. on ${\mathbb R}_{+}$.}.
\end{lemma}
\begin{proof} (see \cite[Lemma 4.1] {Gr2007})  If  $g,h$ $\in$ 	$ L^{\infty}_{2,loc}({\mathbb R}_{+}; X_{+})$, then $B(t,g,h)$ given by  (\ref{6}) can be  written as
\begin{equation*}
\begin{split}
&B(t,g,h)
\\
&=Q^{+}(t,g(t))+a(\|\Lambda g(t)\|) \Lambda g(t) 
-Q^{-}(t,g(t)) 
+F(\| \Lambda g(t) \|, \chi(t,h)) \Lambda g(t),
\end{split}
\end{equation*}	
where   $F(x,y)$ $:=$ $a(x+y)-a(x)$, $x,y \geq 0$,   and  $\chi(t,h)$ $:=$ $\int_0^t\Delta (s,h(s))ds$.  
By (\ref{arep}), $F(x,y)=\int_0^y a'(x+\xi)d\xi$. Recall that  $a'$ is a.e. defined,  positive and non-decreasing  on ${\mathbb R}_{+}$. Thus  $F(x_2,y_2)$ $\geq$ $F(x_1,y_1)\geq 0$,  $\forall$ $x_2\geq x_1\geq 0$ and $y_2\geq y_1\geq0$. This and the isotonicity properties of  $\Lambda$  and $\Delta$  imply that $(g,h) \mapsto F(\| \Lambda g \|, \chi(\cdot,h)) \Lambda g$ is isotone from   $L^{\infty}_{2,loc}({\mathbb R}_{+}; X_{+}) \times L^{\infty}_{2,loc}({\mathbb R}_{+}; X_{+})$ to $L_{1,loc}^1({\mathbb R}_{+}; X_{+})$.
To complete the proof, it is sufficient to remark that $g\mapsto Q^{+}(\cdot, g)+a(\|\Lambda g\|) \Lambda g-Q^{-}(t,g)$ is isotone from   $L^{\infty}_{2,loc}({\mathbb R}_{+}; X_{+})$ to $L_{1,loc}^1({\mathbb R}_{+}; X_{+})$, by virtue of the isotonicity  of $Q^{+}$ and assumption $(A_1)$.
\end{proof}
It appears that  Eq. (\ref{6ab}) could be solved  by  monotone iteration,  Levi's property being useful  to prove the convergence of the iteration. To this end, we introduce a sequence of approximation  solutions to (\ref{6ab}), following ideas of \cite{Gr2007}.\footnote {We proceed as in Step 1 of the proof of  \cite[Theorem 3.1]{Gr2007}, with the difference that, here,  we also approximate the initial datum $f_0$ of (\ref{6ab}).} Specifically, for $f_0\in {\mathcal D}_{+}(\Lambda^2)$ in   (\ref{6ab}), we apply Lemma \ref{lema3}(b) and  choose an increasing sequence $ {\mathcal D}_{+}(\Lambda^\infty) \ni f_{0,n}\nearrow f_0$, as $n\rightarrow \infty $, where the first term of the sequence is  $f_{0,1}=0$. 
Then our approximating sequence is formally given by
\begin{equation}
\begin{split}
f_{1}(t) & =0,\quad f_{2}(t)=V^tf_{0,2},
\\
f_{n}(t) &=V^tf_{0,n}+S(t,f_{n-1},f_{n-2}), \quad t \geq 0; \quad
n=3,4,... .
\end{split}
\label{42nou}
\end{equation}
where
\begin{equation}\label{steq}
S(t,g,h):=\int_0^tV^{t-s}B(s,g,h)ds, \quad t\geq 0.
\end{equation}
From (\ref{Bmap})   and the presence of the integral in the r.h.s of (\ref{steq}), we have
\begin{remark}\label{contS}
 If  $g$, $h$ $\in$ $L^{\infty}_{2,loc}({\mathbb R}_{+}; X_{+})$ then $S(\cdot,g,h)\in C({\mathbb R}_{+};X_{+})$.
\end{remark}

The next three  lemmas serve to show that $f_{n}$ is well defined  and has useful integrability and regularity properties.

For each $k=1,2,...,\infty$, let ${\mathcal M}_k$ be the family of those $g$ $\in$ $C({\mathbb R}_{+};X_{+})$ with the  property that $\forall$ $0<T<\infty$, there is $g_{T}$ $\in$ ${\mathcal D}_{+}(\Lambda^k)$, which may depend only on $g$ and $T$, such that $g(t)$ $\leq$ $g_{T}$ on $[0,T]$.  Obviously, if $k, p=1,2, ,.., \infty$, and $k \leq p$, then   ${\mathcal{M}}_p$ $\subset$ $L_{p,loc}^1({\mathbb R}_{+};X_{+})$ $ \subseteq$ $L_{k,loc}^1({\mathbb R}_{+};X_{+})$. In particular, ${\mathcal M}_\infty$  $ \subset$$L_{\infty,loc}^1({\mathbb R}_{+};X_{+})$.

\begin{lemma}\label{gt}
	
	(a) Under conditions (a) of Theorem \ref{th1}, $S(\cdot,{\mathcal M}_\infty,{\mathcal M}_\infty)$ $\subset$  ${\mathcal M}_\infty$.
	
	(b) Under conditions (b) of Theorem \ref{th1}, $S(\cdot,{\mathcal M}_3,{\mathcal M}_3)$ $\subset$  ${\mathcal M}_3$.
\end{lemma}
\begin{proof}
If  $ g$, $h$ $\in$ ${\mathcal M}_\infty$, in case (a) ($g$, $h$ $\in$ ${\mathcal M}_3$ in case (b)),  then $S(\cdot,g,h)\in C({\mathbb R}_+; X)$, by virtue of Remark \ref{contS}. Moreover, for some  fixed (arbitrary) $0<T<\infty$,   there is $g_T$ $\in$ ${\mathcal D}_{+}(\Lambda^\infty)$, in case (a) ($g_T$ $\in$ $\mathcal{D}(\Lambda^3) \cap X_{+}$, in case (b)) such that $g(t) \leq g_T$  for all $t\in[0,T]$. 
Consequently, using  Lemma \ref{lem2}, the  monotonicity  assumptions of ($A_0$) and ($A_1$), and the obvious inequality $\int_0^t\Delta(s,h(s))ds$ $ \leq \int_0^T\Delta(s,h(s))ds$, we obtain from (\ref{6}) 
\begin{equation}\label{BHqa}
B(t,g,h) \leq u_{g,h,T}, \quad a.e \; on \; [0,T],
\end{equation}
where $u_{g,h,T}(t)$ $:=$ $Q^{+}(t,g_{T})$$ +$$a\left(\| \Lambda g_{T}\| +\int_0^T\Delta(s,h(s))ds\right)\Lambda g_T$.

(a) To prove (a), set 
\begin{equation}\label{sT}
s_{T,a}:=\int_0^T u_{g,h,T}(s)ds.
\end{equation}
Using (\ref{vt}),  (\ref{steq}),  and (\ref{BHqa}), we get $S(t, g,h)\leq s_{T,a}$  a.e. on $[0,T]$. 
But, in case (a),  $g_T $, $\Lambda g_T$ $\in$  ${\mathcal D}_{+}(\Lambda^\infty)$. Besides,  $\Lambda^k Q^{+}(\cdot,g_{T}) \in L^1(0,T;X_{+})$ for all $k=0,1,2,...$, due to assumptions (a) of Theorem \ref{th1}. Thus we have $\Lambda^k u_{g,h,T}$ $\in$ $L^1(0,T;X_{+})$ for all $k=0,1,2,...$. Then, by virtue of  (\ref{com_boc}), for each  $k=1,2,...$,  the operator    $\Lambda^k$  can be applied to (\ref{sT}), and  interchanged   with the integral therein. Therefore,   
$s_{T,a}\in\mathcal{D}_{+}(\Lambda^\infty)$, concluding the proof of (a), because $T$ is arbitrary.

(b) In case (b), recall that   $Q^{+}$ does not  depend explicitly of $t$, because of the assumptions (b)  of Theorem \ref{th1}. Therefore,  $u_{g,h,T}$ 
is  $t$ - independent.  Then, from (\ref{steq}) and  (\ref{BHqa}),  we get $S(t,g,h)$ $\leq$ $ \int_0^tV^{t-s}u_{g,T}ds$ $=$ $ \int_0^tV^{s}u_{g,h,T}ds$ $\leq$ $s_{T,b}$ for all  $0\leq t \leq T$, where 
\begin{equation}\label{sTb}
s_{T,b}:=\int_0^TV^{s}u_{g,h,T}ds.
\end{equation}
Since $g_T$  $\in$ $\mathcal{D}(\Lambda^3) \cap X_{+}$, the domain assumptions on  $Q^{+}$ and $\Lambda$ imply $u_{g,h,T}$   $\in$ $\mathcal{D}(\Lambda^2) \cap X_{+}$. Then by applying (\ref{rezol1}) to (\ref{sTb}), we get  $s_{T,b}$ $\in$ $\mathcal{D}(\Lambda^3) \cap X_{+}$, which completes the proof of (b).
\end{proof}

\begin{lemma}\label{tech}
	
	(a) Assume  conditions (a) of Theorem \ref{th1} hold. Then for each $n=1,2,...$, one has $f_{n}$  $\in$ $ {\mathcal M}_\infty$.

	(b) Assume  conditions (b) of Theorem \ref{th1} hold.  Then for each $n=1,2,...$, one has $f_{n}$ $\in$  ${\mathcal M}_3$.
	
	(c) In both cases, (a) and (b), for each $n=1,2,...$, $f_{n}$ is a.e. differentiable on ${\mathbb  R}_{+}$.
\end{lemma}
\begin{proof}
From (\ref {42nou}), it follows that  the assertions (a) and (b) are trivially  checked for $f_{1}$ and $f_{2}$,   by setting  $g_{1,T}$$:=$$0$ and (due to (\ref {vt}))  $ g_{2,T}$ $:=$$f_{0,2}$, respectively. Since, in  (\ref {42nou}), obviously, $V^{(\cdot)}f_{0,n}\in {\mathcal M}_{\infty}$
the proof of (a) and (b) can be completed  by a straightforward induction based on the application  of Lemma \ref{gt} to (\ref {42nou}).

(c) Cases $n=1$ and $n=2$ are trivial. Let $n=3,4,...$. Fix an arbitrary $0 <T < \infty$. By (a), (b), and, say,  Lemma \ref{lem2},  we have $B(\cdot,f_{n-1},f_{n-2})$, $\Lambda B(\cdot,f_{n-1},f_{n-2})$ $\in$ $ L^{1}(0,T; X_{+})$. But   $\int_0^tV^{t-s}B(s,f_{n-1},f_{n-2})ds$ $=$ $f_{n}(t)$ $-$ $V^{t}f_{0,n}$ $\in$ ${\mathcal D}(\Lambda^3)$ $\subset$  ${\mathcal D}(\Lambda)$ on $[0,T]$.   
Thus, since $L:=-a(\left\| \Lambda f_0\right\|)\Lambda $ is the infinitesimal generator  of the semigroup $V^{t}$, the proof can be easily concluded by a standard argument   (see, e.g., \cite[Ch.4  $\S$ 4.2]{Pazy}), and, finally recalling   that $T$ is arbitrary.
\end{proof}	
\begin{lemma}\label{subgb}  Suppose  the
	 conditions of Theorem \ref{th1} are fulfilled. Then for each 
	 $n=1, 2,3,...$, 
	 
	 (a)  $f_{n}$ $\in$ ${\mathcal M}_{3}$. In particular,  $f_n \in L_{k, loc}^{1}({\mathbb R}_{+}; X_{+})$,  $k=0,1,2,3$;
	 
	 (b) $\Lambda_1 f_n$ $\in$ $L_{loc}^{1}({\mathbb R}_{+};X_+)$;

	 (c) $Q^{\pm}(\cdot,f_{n}(\cdot))\in L_{k, loc}^{1}({\mathbb R}_{+}; X_{+})$ for $k=0,1,2$, and $\Lambda_{1} Q^{\pm}(\cdot,f_n(\cdot))$ $\in$  $L_{loc}^{1}({\mathbb R}_{+};X_+)$.
\end{lemma}
\begin{proof}
(a) is immediate from Lemma \ref{tech} and Remark \ref{lem1a}.

(b) follows from (a) and the property  $\Lambda_1 ({\mathcal M}_{3})$$\subset$ $L^{1}_{loc}({\mathbb R}_{+}; X_{+})$  (due to $(A_3)$).

(c) By (a) and (\ref{a2poz}),  clearly,  $Q^{- }(\cdot ,f_n(\cdot))$ $\in L_{2, loc}^{1}({\mathbb R}_{+};X_{+})$. Moreover,  (\ref{a2poz}) yields $\Lambda_1 Q^{-}(t, f_{n}(t))$ $ \leq$ $ a(\| \Lambda f_{n}(t) \|)\Lambda_1 \Lambda f_{n}(t)$ for a.a.  $t \geq 0$.  But $\Lambda ({\mathcal M}_{3})$ $ \subset$ $ {\mathcal M}_{2}$.   Besides, $\Lambda_1 ({\mathcal M}_{2})$ $\subset$ $L^{1}_{loc}({\mathbb R}_{+}; X_{+})$, by virtue of  $(A_3)$. Consequently,   $\Lambda_1 Q^{-}(\cdot, f_{n}(\cdot))$ $\in$  $L_{loc}^1({\mathbb R}_{+}; X_{+})$.  So $\Lambda_1 Q^{+}(\cdot, f_{n})$ $\in$  $L_{loc}^1({\mathbb R}_{+}; X_{+})$, by virtue of (\ref{lambdaunu}). Also, from (\ref{povg}), we get $Q^{+}(\cdot ,f_n(\cdot))$ $\in L_{2, loc}^{1}({\mathbb R}_{+};X_{+})$.  Remark \ref{lem1a} completes the proof of (c).
\end{proof}
The  next lemmas are needed to establish the convergence of the approximating sequence  $\{f_{n}(t)\}$ defined by (\ref{42nou}).
\begin{lemma} \label{lemmonot} The  sequence $\{f_{n}(t)\}$  is positive and increasing for all $t \geq 0$.
\end{lemma}
\begin{proof} This result follows from (\ref{42nou}) by a straightforward induction which uses the positivity and monotonicity of  $\{f_{0.n}\}$, the positivity of the linear semigroup  $V^t$, and  applies  Lemma \ref{lem2}.
\end{proof}

Based on Lemma \ref{tech}(c),  we can differentiate (\ref{42nou}),   a.e. on ${\mathbb  R}_{+}$. We obtain
\begin{equation}
\frac d{dt}f_n(t)=B(t,f_{n-1},f_{n-2})-a(\left\| \Lambda f_0\right\|
)\Lambda f_n(t),\quad t>0\quad a.e., \quad n\geq 3.  \label{dfqdtnou}
\end{equation}
Integrating again Eq. (\ref{dfqdtnou}), and using (\ref{6}), we obtain for $\ n\geq 3$,
\begin{equation}
\begin{split}
& f_n(t)  =f_{0,n}+\int_0^t [Q^{+}(s,f_{n-1}(s))-Q^{-}(s,f_{n-1}(s))]ds \\
&+  \int_0^t \biggl[ a\left( \left\| \Lambda f_{n-1}(s)\right\|
+\int_0^s\Delta (\tau,f_{n-2}(\tau))d\tau
\right) \Lambda f_{n-1}(s)\\
&-a(\left\|\Lambda f_0\right\| )\Lambda f_n(s)\biggr]ds,
\end{split}
\label{p1nou}
\end{equation}
which is useful to prove the following property.
\begin{lemma}
\label{Lemma 5new} For all  $n=2,3,....$,
\begin{equation}
f_n(t)+\int_0^t Q^{-}(s,f_{n-1}(s)) ds\leq f_{0,n}+\int_0^t Q^{+}(s,f_{n-1}(s))ds, \quad  \forall t \geq 0, \label{inqnew}
\end{equation}
and
\begin{equation}
\left\| \Lambda f_n(t)\right\| +\int_0^t\Delta (s,f_{n-1}(s))ds \leq \left\|
\Lambda f_{0,n}\right\| \leq \left\|
\Lambda f_0\right\|, \quad  \forall t \geq 0.  \label{42bnew}
\end{equation}
\end{lemma}
\begin{proof} (See the proof of Lemma 4.3 in \cite{Gr2007}.) We proceed by induction.  Using (\ref{vt}),  one finds from (\ref{42nou}) that $0=f_{1}(t) \leq f_{2}(t)\leq f_0$. But $\Delta (t,0)=0$ a.e., because of  (\ref{qzero}). Therefore, (\ref{inqnew}) and (\ref{42bnew}) are trivially verified for $n=2$.

Suppose that (\ref{inqnew}) and (\ref{42bnew}) are true for $n=2,3,...,q$.  Since    $0\leq \Lambda f_{q}(t)\leq \Lambda f_{q+1}(t)$ (as a consequence of   Lemma \ref{lemmonot}) and $a$ is non-decreasing, we get
\[
0 \leq  a\left( \left\| \Lambda f_{q}(s)\right\| +\int_0^s\Delta
(\tau ,f_{q-1}(\tau))d\tau \right) \Lambda f_{q}(s) \leq a\left( \left\| \Lambda f_0\right\|
\right)  \Lambda f_{q+1}(s),
\]
which can be applied to   (\ref{p1nou}) in the case $n=q+1$. Thus we obtain
\begin{equation}
f_{q+1}(t)+\int_0^t Q^{-}(s,f_{q}(s))ds\leq f_{0,q+1}+\int_0^t Q^{+}(s,f_{q}(s))ds, \quad  \forall t \geq 0,
\label{inqnew2}
`\end{equation}
concluding the induction argument for the validity of (\ref{inqnew}).
Further, based on Lemma  \ref{subgb},  we can  apply  $\Lambda$  to (\ref{inqnew2})  and use  (\ref{com_boc}). Then, applying, (\ref{1amonot}), (\ref{1a}),  (\ref{bi}) and (\ref{3b}), we  finally obtain  (\ref{inqnew}) for $n=q+1$, which concludes the  proof of the lemma.
\end{proof}
\begin{lemma}
\label{5pova} Under the conditions of Theorem \ref{th1},
\begin{equation}
\left\| \Lambda^2 f_n(t)\right\| \leq \left\| \Lambda^2 f_0\right\| \exp (\rho(\left\| \Lambda
_1f_0\right\| )t), \quad \forall t \geq 0, \quad
 n=1,2,....  \label{pov2}
\end{equation}
\end{lemma}
\begin{proof} (See the proof of Lemma 4.4 in \cite{Gr2007}.)
	Cases $n=1$ is trivial. For $n \geq 2$, based on Lemma \ref{subgb}, we apply   $\Lambda_{1}$ to (\ref{inqnew}) and  use   (\ref{1amonot}), (\ref{1a}), and (\ref{bi}). We get 
\[
\left\| \Lambda_{1} f_n(t)\right\| \leq \left\|
\Lambda_{1} f_{0,n}\right\|-\int_0^t(\left\Vert \Lambda_{1} Q^{-}(t,f_{n-1}(s))\right\Vert -\left\Vert \Lambda_{1}
Q^{+}(t,f_{n-1}(s))\right\Vert)ds
\]
for all $n=2,3,...$, and $t \geq 0$.
But in the above inequality, the integrand is positive, because of (\ref{lambdaunu}). Thus
\begin{equation}
\left\| \Lambda_{1} f_n(t)\right\| \leq\left\|
\Lambda_{1} f_{0,n}\right\| \leq \left\|
\Lambda_{1} f_0\right\|,\quad \forall t \geq 0, \quad
n=2,3,....    \label{42bnewbis}
\end{equation}
Based on  Lemma \ref{subgb},  we apply (\ref{com_boc}) (with $\Gamma=\Lambda^2$) to (\ref{inqnew}). Then, using (\ref{1amonot}), (\ref{1a}),  (\ref{bi}), and  (\ref{povg}), we obtain
\[
\left\| \Lambda^2 f_n(t)\right\|\leq \left\| \Lambda^2 f_{0,n}\right\|+ \int_0^t\rho ((\left\Vert \Lambda _{1}f_{n-1}(s)\right\Vert
)\left\Vert \Lambda ^{2}f_{n-1}(s)\right\Vert ds, \; \forall t \geq 0
\]
($n=2,3,....$).  But $\rho ((\left\Vert \Lambda _{1}f_{n-1}(s)\right\Vert
)\leq \rho ((\left\Vert \Lambda_{1}f_0 \right\Vert
)$,  by virtue of  (\ref{42bnewbis}). Therefore,
\[
\left\| \Lambda^2 f_n(t)\right\|\leq   \left\| \Lambda^2 f_{0}\right\|+ \rho ((\left\Vert \Lambda _{1}f_0\right\Vert)\int_0^t\left\Vert \Lambda ^{2}f_n(s)\right\Vert ds \; \forall t \geq 0; \; n=2,3,...,  \label{it1}
\]
where the use of Gronwall's inequality concludes
the proof.
\end{proof}

By Lemmas  \ref{lemmonot}  and \ref{5pova}, for every $t\geq 0$, the positive sequence  $\{\Lambda^2 f_n(t)\}$ is   increasing, and  norm bounded, respectively.  Then Lemma \ref{lema3}(c) applies. Therefore,  $\exists$ $f:{\mathbb  R}_+ \mapsto {\mathcal D}_{+}(\Lambda^2)$ measurable,  such that, $\forall t \geq 0$,
\begin{equation}\label{linf}
\Lambda^k f_{n}(t) \nearrow \Lambda ^kf(t),  \quad as \quad n \rightarrow \infty; \quad  k=0,1,2.
\end{equation}
In particular, $\Lambda^k f_{0,n}\nearrow \Lambda^k  f_0$, $k= 1,2$. Then taking the limit in (\ref{pov2}), we find that $f$ satisfies (\ref{7c}). This and  Remark \ref{lem1a} imply  $f$ $\in$ $L^1_{k,loc} ({\mathbb  R}_{+}; X_{+})$, $k=0,1,2$.  Moreover,  by (\ref{a54}), we find that  $Q^{\pm}(\cdot, f)$, $\Lambda Q^{\pm}(\cdot, f)$, and $ \Delta f$ are in $ L^1_{loc} ({\mathbb  R}_{+}; X_{+})$. 
Further,  from ({\ref{a54}}) and (\ref{linf}), it follows that the increasing sequences  $\{\Lambda^k Q^{\pm}(t, f_n)\}$, $k=0,1$, are bounded in norm. Therefore, Levi's property implies that  they are  convergent a.e. on  ${\mathbb  R}_{+}$.   But  $ Q^{+}(t, \cdot)$ and $Q^{-}(t, \cdot)$ are o-closed.  Then, $Q^{\pm}(t, f_n) \nearrow Q^{\pm}(t, f)$ as $n \rightarrow \infty$  a.e. on ${\mathbb  R}_{+}$. Since $\Lambda$ is closed, it also follows that  $\Lambda Q^{\pm}(t, f_n) \nearrow \Lambda Q^{\pm}(t, f)$ as $n \rightarrow \infty$, a.e. on ${\mathbb  R}_{+}$.  Consequently, $\Delta (t, f_n) \nearrow \Delta (t, f)$ as $n\rightarrow \infty$, a.e. on ${\mathbb  R}_{+}$. Thus, applying  the Lebesgue's dominated convergence theorem, we obtain $ \int_0^s\Delta (\tau,
f_n(\tau))d\tau$ $\nearrow$ $ \int_0^s\Delta (\tau,
f(\tau))d\tau$  as $n \rightarrow \infty$, $\forall s>0$.
By the above considerations, and using that  $a$ is non-decreasing and continuous, we are enabled to apply conveniently the  dominated convergence theorem in (\ref{p1nou}) and (\ref{42bnew}). It follows that  $f$ is solution to (\ref{35}), and satisfies (\ref{7c}). Finally, Proposition \ref{equiv} concludes the existence part of Theorem \ref{th1}.
 
 The uniqueness of the solution follows by the same argument as in  Ref.  \cite{Gr2007} (inspired from  \cite{Ar})).
 In detail,  since $f$ is the limit of $\{f_n\}$, by applying the dominated convergence theorem to (\ref{42nou}), we find that $f$ is also solution to Eq. (\ref{6ab}). Let $F$ be another continuous, positive solution of  (\ref{3}), with the properties stated in Theorem \ref{th1}. Thus $F$ satisfies (\ref{5}), as $f$ does. Obviously, $F$ is also a solution to (\ref{6ab}), and  a simple induction  implies  $ f_{n}(t) \leq F(t)$ for all $t \geq 0$, and $n=1,2...$.  Consequently,  $0 \leq f(t) \leq F(t) $ for all $t \geq 0$. Thus,  if $\exists$   $t_* >0$ such that  $F(t_*) \neq f(t_*)$, then  $f(t_*) < F(t_*)$, hence   $ \|\Lambda f(t_*)\| < \|\Lambda F(t_*)\|$. Since $\Delta$ is isotone, we get
\[
\|\Lambda f_0 \|= \|\Lambda f(t_*)\| +\int_0^{t_*} \Delta (s,f(s))ds < \| \Lambda F(t_*) \| +\int_0^{t_*} \Delta (s,F(s))ds,
\]
in contradiction with the fact that  $F$ satisfies (\ref{5}).
$\Box$
\section{Conclusions}
In the present work, we have revised and generalized  the main results of Ref. \cite{Gr2007}.  If  the  setting of the current note reduces  to an AL - space,  then, in essence,   Theorem \ref{th1} and Corollary \ref{cor1}  reduce to \cite[Theorem 3.1]{Gr2007} and \cite[Corollary 3.1]{Gr2007}, respectively. (In fact, Theorem \ref{th1}(b) slightly improves the result of  \cite[Theorem 3.1(b)]{Gr2007}, by relaxing the conditions on the initial datum.) An analogue of \cite[Proposition 3.1]{Gr2007} can be also obtained by  adapting directly its proof to the present context.

The analysis of this paper is close to  the argument of  \cite{Gr2007}, but technically there are some differences.  Indeed, in Ref. \cite{Gr2007},    (Theorem 3.1(a))  was  obtained by a two-step demonstration. In the first step ("Step 1") the theorem was proved  for an initial datum  in  ${\mathcal D}_{+}(\Lambda^\infty)$ (in the setting of an AL-space)\footnote{${\mathcal D}_{+}(\Lambda^\infty)$  appears in \cite{Gr2007} as ${\mathcal D}^\infty_{+}(\Lambda)$. }. This was done by approximating the solution  of  \cite[Eq. (1.1)] {Gr2007} by a sequence similar to that defined by (\ref{42nou}), in the previous section, but keeping the initial datum fixed in ${\mathcal D}_{+}(\Lambda^\infty)$.   The purpose of  the second step ("Step 2")  was  to extend the result of "Step 1", by considering an  initial datum  $f_{0}$   $\in{\mathcal D}_{+}(\Lambda^2)$. Thus the solution of \cite[Eq. (1.1)]{Gr2007} was approximated by a sequence, denoted  $\{F_{i}\}$ in \cite{Gr2007}, of solutions of the same equation, but  corresponding to an increasing sequence of initial data in ${\mathcal D}_{+}(\Lambda^\infty)$,  converging to the original $f_{0}$. Then "Step 2" was concluded, based on the assertion that  $\{F_{i}\}$ is increasing. Unfortunately, the monotonicity of  $\{F_{i}\}$ has been erroneously justified in Ref. \cite{Gr2007}, so the proof of Theorem 3.1(a) is incomplete  therein. However, the error can be easily corrected by reconstructing  $\{F_{i}\}$ to approximate \cite[Eq. (3.15)]{Gr2007} instead of \cite[Eq. (1.1)]{Gr2007}, as was done in \cite{Gr2007} (see the   Appendix below).  

Finally, it should be emphasized that in the present note,
   the two-step proof of \cite[Theorem 3.1(a)]{Gr2007} has been reduced to a modified version of "Step 1" of that proof. This was done  by  introducing (\ref{42nou}) as  a diagonalization, in some sense, of the main approximation scheme used in \cite{Gr2007}.

\bigskip
\begin{center}
APPENDIX
\end{center}
\begin{center}
  \small{{\bf Corrigendum} to "A Nonlinear Evolution Equation in an Ordered Space, Arising from Kinetic Theory" [Commun. Contemp. Math.,  09 (2007) 217--251]}
\end{center}
Although Theorem 3.1 of Ref. \cite{Gru2007ap} is true as stated therein, the proof  of part (a) of the  theorem (\cite[Theorem 3.1(a)]{Gru2007ap}) is  incomplete, because of an  easily correctable error occurred in "Step 2" of that proof. Specifically, at page 235 in \cite{Gru2007ap},   it is asserted without demonstration (and, seemingly, incorrectly)  that the sequence $\{F_i\}$,  introduced at the same page, is {\it increasing}. Therefore, the assertion  cannot be  used  to justify a possible convergence of  $\{F_i\}$, as was done in \cite[p. 235]{Gru2007ap}.
Below, we  repair the  error, by redefining the sequence $\{F_i\}$, and modifying the argument of "Step 2" accordingly, to fill the gap in the proof (\cite[Theorem 3.1(a)]{Gru2007ap}).    To this end, we refer to paper \cite{Gru2007ap}, using the same notation as there,  and provide the following:

{\bf Step 2 of the proof of Theorem 3.1(a) in \cite{Gru2007ap}  p. 235 (corrected version)}: Let $f_0$ be as Theorem 3.1, i.e., $f_0\in {\mathcal D}(\Lambda ^2)\cap X_{+}$. By Lemma 2.1(b), there is increasing sequence $\{f_{0,i}\}_{i=1}^{\infty}$ $\subset$ 
${\mathcal D}^\infty_{+}(\Lambda)$, such that $f_{0,i}\nearrow f_{0}$   as $i\rightarrow \infty$.
Similarly to  (3.16), for each $i=1,2,...$, define 
\begin{align*}
\begin{aligned}
f_{1,i}(t) & =0,\quad f_{2, i}(t)=V^tf_{0,i}, \\
f_{n, i}(t) & = V^tf_{0,i}+\int_0^tV^{t-s}B(s,f_{n-1,i},f_{n-2, i})ds, \quad (t \geq 0); \quad
n=3,4,...,
\end{aligned}
\end{align*}
where  $\{V^t\}_{t\geq 0}$ is the positive $C_{0}$ semigroup with  infinitesimal generator $L=-a(\|\Lambda f_0\|)\Lambda$. Using  Lemma 4.1 in  the definition of $f_{n,i}$ one gets easily 
\[
f_{n,i}(t)\leq f_{m,j}(t) \quad (t\geq 0) \tag{$\star$}
\]
for all  $n,m,i,j=1,2,...$, such that $n\leq m $ and $i\leq j$.
Proceeding  as in Step 1 (to arrive at Eq. (4.25)), one finds that for each fixed $i$, the sequence $\{f_{n,i}\}_{n}$ is increasing, and $\exists$ $F_{i}:{\mathbb  R}_{+} \mapsto {\mathcal D}(\Lambda ^2)\cap X_{+}$ such that  $f_{n,i}(t)\nearrow {F}_{i}(t)$,  as $n\rightarrow \infty$, $ t\geq 0$, and  $\Lambda^k F_{i}$,  $Q^{\pm}(\cdot, F_{i})$, $\Lambda Q^{\pm}(\cdot, F_{i})$ $\in$ $L^1_{loc}({\mathbb  R}_{+};X_{+})$, $k=0,1,2$.
Moreover $F_{i}$ satisfies an equation of the form (4.25), with initial datum $f_{0,i}$,
\begin{align*}
\begin{aligned}
{F}_{i}(t) &  =f_{0,i}+\int_0^tQ(s,{F}_{i}(s))ds  \\
 + & \int_0^t\left[ a\left( \left\| \Lambda {F}_{i}(s)\right\| +
\int_0^s\Delta (\tau,
{F}_{i}(\tau))d\tau \right) -a(\left\| \Lambda f_0\right\| )\right] \Lambda
{F}_{i}(s)ds, \quad t \geq 0.
\end{aligned}
\tag{$\star \star$}
\end{align*}
Also, $\|\Lambda F_{i}(t)\| \leq \|\Lambda f_{0}\|$, and $\left\| \Lambda ^2F_i(t)\right\| \leq \exp (\rho (\left\| \Lambda
_1f_0\right\| )t)\left\| \Lambda ^2f_0\right\|$, $ (t\geq 0)$. Besides,  ($\star$) implies that  $\{F_i(t)\}_{i}$ is increasing for all $t\geq 0$. Since $X$ is monotone complete,   $\exists$ $f \in L^{1}_{loc}({\mathbb  R}_{+},X_{+})$ such that   $F_i(t) \nearrow f(t)$ as $i\rightarrow \infty$,     $\|\Lambda f(t)\| \leq \|\Lambda f_{0}\|$, and $\left\| \Lambda ^2f(t)\right\| \leq \exp (\rho (\left\| \Lambda
_1f_0\right\| )t)\left\| \Lambda ^2f_0\right\|$ for all $ t\geq 0$. 
Thus one can apply the dominated convergence theorem to ($\star\star$). One gets an equation for  $f$, as Eq. (4.25), but with  $f_0$  in ${\mathcal D}(\Lambda ^2)\cap X_{+}$, and not in ${\mathcal D}^\infty_{+}(\Lambda)$, as  assumed in Step 1. To conclude the existence part of the proof of Theorem 3.1(a), one  reasons  as in the last part of Step 1  (after Eq. (4.25)). 

\end{document}